\documentclass[10pt]{amsart}
\usepackage{amsmath}
\usepackage{amsfonts}
\usepackage{amsthm}
\usepackage{amssymb}
\usepackage{mathrsfs}
\usepackage{hyperref}
\usepackage{array}
\usepackage{latexsym}
\usepackage{verbatim}
\usepackage{a4wide}
\usepackage{enumerate}
\usepackage{lipsum}
\usepackage[all]{xy}
\usepackage{tikz-cd}  
\usepackage{xcolor}
\usepackage{ulem}
\numberwithin{equation}{section}

\renewcommand{\thetheoremName}

\newcommand{\IC}{\mathbb{C}}
\newcommand{\IF}{\mathbb{F}}

\newcommand{\IT}{\mathbb{T}}
\newcommand{\IN}{\mathbb{N}}
\newcommand{\IG}{\mathbb{G}}

\newcommand{\IQ}{\mathbb{Q}}

\newcommand{\IZ}{\mathbb{Z}}

\newcommand{\calF}{\mathcal{F}}

\newcommand{\calO}{\mathcal{O}}

\newcommand{\calS}{\mathcal{S}}

\newcommand{\im}{\mathfrak{m}}

\newcommand{\Z}{\IZ}
\newcommand{\Zp}{\IZ_p}
\newcommand{\Zpx}{\IZ_p^\times}
\newcommand{\Cpx}{\IC_p^\times}
\newcommand{\Fpx}{\IF_p^\times}

\newcommand\robout{\bgroup\markoverwith {\textcolor{blue}{\rule[0.5ex]{2pt}{0.4pt}}}\ULon} 
\newcommand{\ds}{\displaystyle}

\def\Hom{\mathrm{Hom}}

\def\Spec{\mathrm{Spec}}
\def\Spf{\mathrm{Spf}}

\def\wt{\widetilde}

\DeclareMathOperator\ord{ord}

\DeclareMathOperator\Fil{Fil}

\newtheorem{theorem}{Theorem}[section]
\newtheorem*{theorem*}{Theorem}

\newtheorem{lemma}[theorem]{Lemma}
\newtheorem{prop}[theorem]{Proposition}

\theoremstyle{definition}

\theoremstyle{remark}
\newtheorem{example}[theorem]{Example}
\newtheorem{remark}[theorem]{Remark}

\thanks{The first author's research was partially funded by NSF grant DMS-1702178.  The second author's research was partially funded by START-prize Y-966 of the Austrian Science Fund (FWF) under P.I. Harald Grobner.}

\begin{document}
\title{Non-vanishing of critical $L$-values in Hida families
}
\author{Robert Pollack and Vlad Serban}

\bibliographystyle{alpha}

\begin{abstract}
We study the vanishing of $L(f,\chi,j)$ as $f$ runs through all classical forms in a $p$-adic Hida family (including forms with arbitrarily high nebentype at $p$), $\chi$ runs through all characters of $p$-power conductor, and $j$ is a critical value.  We show that if infinitely many of these $L$-values vanish (apart from the ones forced to vanish by the sign of their functional equation) then this infinitude of vanishing must be exceptionally regular, so regular in fact that one can typically rule out this possibility in any given example.  Indeed, we systematically verified that such regular vanishing does not occur in multiple Hida families twisted by a wide range of quadratic characters by computing the corresponding two-variable $p$-adic $L$-functions via overconvergent modular symbols.

\end{abstract}

\maketitle
\section{Introduction}
Let $f$ denote a cuspidal Hecke eigenform of weight $k\geq 2$. The critical values of the complex $L$-function associated to $f$, namely $L(f,\chi,j)$ for $1\leq j\leq k-1$ and $\chi$ a Dirichlet character, are of particular interest due to their connection to arithmetic information arising from $f$ as predicted by the Bloch-Kato conjectures. 
As the non-vanishing of these values relate to ranks of elliptic curves, Selmer groups and other crucial arithmetic data, the non-vanishing of such $L$-values is a well-studied question (see, e.g., \cite{Rohrlich} and also the survey \cite{Vatsalmodp} for non-vanishing modulo $p$). 

In this note, we study the non-vanishing of such critical $L$-values in Hida families; that is, we study the $L$-values $L(f,\chi,j)$ where $f$ runs through a fixed Hida family, $\chi$ runs through characters of $p$-power conductor and $j$ is a critical value.  In \cite{Greenberg}, Greenberg speculates that in this family of $L$-values there is only a finite amount of vanishing apart from the $L$-values forced to vanish due to the sign of their functional equation.

In tackling this question, our main tool is the two-variable $p$-adic $L$-function.  For a fixed ordinary eigenform $f$, its associated  $p$-adic $L$-function (as in \cite{MTT}) interpolates the algebraic parts of $L(f,\chi,j)$ for $\chi$ of $p$-power conductor and $j$ critical.  There is also a two-variable $p$-adic $L$-function which interpolates all of these values as $f$ runs through all forms in a fixed Hida family. It was constructed by Kitigawa \cite{Kitigawa} building on work of Mazur and separately by Greenberg and Stevens \cite{GSperiods93}.

When the Hida family has rank 1 over the Iwasawa algebra, the two-variable $p$-adic $L$-function can be expressed as $p-1$ power series in two variables (the weight variable and the cyclotomic variable) and thus each such power series can be seen as a function on the unit two-dimensional polydisk.  In this light, our non-vanishing question amounts to whether or not the two-variable $p$-adic $L$-function vanishes infinitely often at the special points in this polydisk which correspond to the classical $L$-values under consideration. (See Remark \ref{rmk:bigrank} for a discussion about the case when this rank is bigger than 1.)\par 

 One of the main reasons this non-vanishing question is a non-trivial one is that one must simultaneously deal with infinitely many possible $p$-power order nebentypes in the Hida family alongside infinitely many possible $p$-power order characters in the cyclotomic variable. However, it turns out $p$-adic formal power series are just rigid enough to prove an \emph{unlikely intersection result} in this context, namely to prove that a component of the $p$-adic $L$-function vanishing along infinitely many \emph{special values} has to have a \emph{special} shape. More precisely, if we pretend that $k,j$ are fixed and thus only vary the nebentypes and $p$-power character, our special values are essentially the torsion points on the $p$-adic polydisk of the formal torus  $\widehat{\IG}_m^2$ over $\IZ_p$. If we moreover pretend that the $p$-adic $L$-function is just a two-variable polynomial, it follows from a classical result known as the Manin--Mumford conjecture (see e.g., \cite{Zannierbook}) that the $p$-adic $L$-function must vanish along a translate of a one-dimensional subtorus by some torsion point. Such Manin--Mumford type theorems were extended in work of the second author \cite{Serban:2016aa,SerbanpadicMM,Serbanunlikely} to the $p$-adic analytic functions in question, so that building on those results we can deal with the full set of critical values. We thus show in section \ref{section:rigidity}: 

\begin{theorem}
\label{thm:mainrigidityintro}  
Consider one of the finitely many branches of a two-variable $p$-adic $L$-function expanded as $L_{p,i}(w,T)\in \IZ_p[[w,T]]$ for a fixed choice of topological generator $\gamma$ of $1+p\IZ_p$ (or of $1+p^2\IZ_p$ if $p=2$). Assume $L_{p,i}$ vanishes at infinitely many critical $L$-values for the interpolated cuspidal eigenforms. Then $L_{p,i}(w,T)$ must be, possibly over a finite extension $\IZ_p[\xi]$, divisible by a factor of the form
$$\gamma^{M}(w+1)^N-\xi(T+1)\text{ or }\gamma^{M}(w+1)-\xi(T+1)^N$$ for some fixed $M,N\in \IZ_p$ and a fixed $p$-power root of unity $\xi$. 
\end{theorem}

Theorem \ref{thm:mainrigidityintro} shows that an infinite amount of vanishing implies vanishing along a full formal subtorus translate (and hence vanishing at \emph{every} critical value on that component) which would seem quite surprising unless forced by the sign of the functional equation. We note that rigidity results for $p$-adic power series of the same flavor were also established and utilized for interesting arithmetic applications by Monsky \cite{MR614398}, Hida \cite{HidaHeckeFieldsGrowth} and Stubley \cite{Stubley}. 

\par

Finally, we combine our rigidity results with explicit computations of two-variable $p$-adic $L$-functions via overconvergent modular symbols to show that such vanishing along a formal subtorus does not occur.  Specifically, our rigidity results imply that 
the algebraic part of one of the following $L$-values:
$$
L(f_{2,\chi},\chi^a,1) \mbox{ and } L(f_{2,\chi^a},\chi,1)
$$
is divisible by $p$ where $\chi$ is a fixed primitive character and $\chi^a$ varies over all powers of $\chi$.  Here $f_{2,\epsilon}$ is the specialization of our Hida family to weight 2 with nebentype $\epsilon$.  
\par

Finding the valuation of all of these $L$-values is a finite computation which can be efficiently done with overconvergent modular symbols (see section \ref{sec:computations}).  
We examined three Hida families twisted by hundreds of quadratic characters, and in all of these examples, the valuations of these $L$-values decreased as the conductor of $\chi$ increased.  In every example we were ultimately able to find $\chi$ with large enough conductor so that {\it all} of the above $L$-values were {\bf not} divisible by $p$ implying that only finitely many classical $L$-values in the twisted Hida family vanish and obtained:

\begin{theorem}\label{thm:maincomputational}
Besides the $L$-values forced to vanish because of the sign of the functional equation, the vanishing 
$$L(f_{k,\epsilon}\otimes \chi_D,\chi_p,j)=0$$
holds for only finitely many such $L$-values, where:\begin{itemize}
    \item the cusp forms $f_{k,\epsilon}$ are allowed to vary through all the specializations to weight $k\geq 2$ and 
    $p$-power nebentype $\epsilon$ on the ordinary $p$-adic families of tame level $N$ for $(p,N)\in\{(3,5), (5,3), (11,1)\}$ which lie above the component of weight space containing weight 2,
    \item $\chi_D$ runs through all quadratic characters of discriminant $D$ coprime to $pN$ with $|D|<350$ for $p=3$ and $|D|<500$ for $p=5,11$.
    \item $\chi_p$ runs through all the characters of $p$-power conductor and where $1\leq j\leq k-1$. 
\end{itemize}
\end{theorem}

It is clear that the computations and results such as Theorem \ref{thm:maincomputational} can be extended. It is however unclear at the present time whether this approach guarantees that the finite vanishing result, when true for a particular Hida family, can be established by a finite amount of computation. Indeed, one cannot a priori exclude that the $p$-adic $L$-function contains a toroidal factor modulo $p$ which does not lift to characteristic zero. However, we have yet to encounter such an exception in practice.

\section*{Acknowledgements}
We would like to thank Frank Calegari and Rob Harron for helpful conversations on the topics of this paper. 

\section{Background on $p$-adic $L$-functions}

Fix a prime $p$ and let $\IT$ denote a local summand of the Hida-Hecke algebra with tame level $\Gamma_0(N)$. We assume that $\IT$ is rank 1 over the Iwasawa algebra and thus the eigenforms in this Hida family are in one-to-one correspondence with the points of a fixed component of weight space.  
The classical specializations of this Hida family correspond to homomorphisms $\varphi :\ \Zpx \to \Cpx$ of the form $x \mapsto x^{k-2} \epsilon(x)$ with $k \geq 2$ is an integer,  $\epsilon$ is a Dirichlet character of conductor $p^n$ such that $\epsilon(-1) = (-1)^k$, and $\varphi|_{\Fpx} = \omega^r$.   Here $r$ is in a fixed residue class modulo $p-1$ corresponding to our fixed component of weight space.  Write $f_{k,\epsilon}$ for the unique normalized eigenform corresponding to the specialization at $\varphi$.

For any of these eigenforms $f = f_{k,\epsilon}$, there is an associated $p$-adic $L$-function (as in \cite{MTT}) which interpolates the special values:
$$
\frac{L(f,\chi,j)}{(-2\pi i)^{j-1} \Omega_f^\pm} \in \overline{\IQ}
$$
where $1 \leq j \leq k-1$, $\Omega_f^\pm$ is a canonical period of $f$,  and the sign of $\Omega_f^\pm$ is given by the sign of $(-1)^{j-1}\chi(-1)$.  That is, there exists a measure $\mu_f$ on $\Zpx$ such that 
$$
\int_{\Zpx} x^{j-1} \chi(x)~ d\mu_f(x) =
\frac{1}{a_p(f)^n} \cdot \frac{p^{nj} (j-1)!}{\tau(\chi^{-1})} \cdot \frac{L(f, \chi^{-1},j)}{(-2\pi i)^{j-1} \Omega_f^\pm}.
$$
Here $\chi$ is a Dirichlet character of conductor $p^n > 1$ and $\tau$ denotes a Gauss sum.

In this paper, we will primarily be interested in the power series representation of the $p$-adic $L$-function in the ``$T$-variable".   To this end, set $q=p$ if $p>2$ and $q=4$ for $p=2$.  Choose $\gamma$ a topological generator of $1+q\Z_p$. Then for each fixed $i \pmod{p-1}$, there exists a power series $L_{p,i}(f,T) \in \Zp[[T]]$ such that 
$$
L_{p,i}(f, \gamma^{j-1}\zeta-1) = \int_{\Zpx} \langle x \rangle^{j-1} \chi_\zeta(x) \omega^i(x) ~d\mu_f(x)
$$
where $1\leq j \leq k-1$, $\zeta$ is a $p$-power root of unity, and $\chi_\zeta$ is the character of $\Zpx$ which factors through $1+q\Zp$ and sends $\gamma$ to $\zeta$.  We refer to this as the $i$-th branch of the $p$-adic $L$-function---it is the one that records the values $\int_{\Zpx} \chi ~d\mu_f$ for characters $\chi$ whose restriction to $\Fpx$ equals $\omega^i$.  We note that the $T$-variable gives a non-canonical representation of the $p$-adic $L$-function depending on our choice of $\gamma$ as a topological generator of $1+q\Zp$.

These one-variable $p$-adic $L$-functions interpolate into a two-variable $p$-adic $L$-function along the Hida family (as in \cite[section 3.5]{EmertonPollackWeston}).  That is, for each fixed $i \pmod{p-1}$, there is a power series $L_{p,i}(w,T) \in \Zp[[w,T]]$ such that 
for each classical specialization $f_{k,\epsilon}$ we have
$$
L_{p,i}( \gamma^{k-2}\epsilon(\gamma)-1,T) = L_{p,i}(f,T).
$$
This representation of the two-variable $p$-adic $L$-function in the ``$w$-variable" is again non-canonical and depends on our choice of $\gamma$.

\section{Background on non-vanishing of $L$-values in Hida families}
\label{sec:back_nonvanish}

As $f_{k,\epsilon}$ runs through the eigenforms in a fixed Hida family, one can ask about the vanishing or non-vanishing of the values $L(f_{k,\epsilon},\chi,j)$ where $\chi$ runs through the $p$-power conductor characters and $1 \leq j \leq k-1$.

A first fact to state is that only the central values can vanish.

\begin{prop}\label{prop:onlycentralvanishing}
Let $f$ be a cuspidal eigenform in $S_k(\Gamma_1(M),\epsilon)$ and let $j$ be an integer such that $1 \leq j \leq k-1$.  If $L(f,\chi,j) = 0$, then $j = k/2$.
\end{prop}

\begin{proof}
A proof of this fact is given by Rob Harron on MathOverflow \cite{Harron}.
The basic idea is that the Euler product defining the $L$-series converges for the real part of $s$ bigger than or equal to $\frac{k+1}{2}$ and is thus non-zero is this region.  The functional equation then gives the non-vanishing for the real part of $s$ smaller than or equal to $\frac{k-1}{2}$.
\end{proof}

Thus, we are left to consider the non-vanishing of the values $L(f_{\epsilon,k},\chi,k/2)$ for $k$ even.  One could ask if all but finitely many of these values are non-zero as $\epsilon$, $\chi$ and $k$ vary, but immediately one sees that this cannot hold in general simply because signs of functional equations can force infinitely many of these values to vanish.  That is, consider the $L$-values of the form:\ 
$$
L(f_{\epsilon^2,k},\epsilon^{-1},k/2).
$$
As the form $f_{\epsilon^2,k} \otimes \epsilon^{-1}$ has trivial nebentype, its functional equation relates the above $L$-value back to itself and is thus forced to vanish when the sign of the functional equation is $-1$. 

These particular values are interpolated by exactly two of the branches of the two-variable $p$-adic $L$-function, namely, $L_{p,i}(w,T)$ where $i = \frac{r}{2}$ or $i = \frac{r}{2} + \frac{p-1}{2}$.  Both of these branches have their own functional equations (with possibly different signs).  But if either has a sign of $-1$ then all of the values interpolated by that branch vanish (with finitely many exceptions caused by trivial zeroes of the $p$-adic $L$-functions).  See \cite[section 4]{Howard} for more details.

When the sign of the functional equation is $+1$ (and thus the above constraint is not present) Greenberg conjectures \cite{Greenberg} that all but finitely many of these values are non-zero.  When the sign of the functional equation is $-1$, Greenberg conjectures that $L'(f_{\epsilon^2,k},\epsilon^{-1},k/2)$ is non-zero for all but finitely many exceptions.\footnote{In fact, in \cite{Greenberg}, Greenberg only conjectures this non-vanishing when $k=2$, but in the literature this more general conjecture is often attributed to him.}

We note that in the $+1$ sign case, one can verify this conjecture by finding a single non-zero value.  Indeed, 
$L_{p,i}((1+T)^2-1,T)$ is an Iwasawa function which interpolates all of these values.  If one value is non-zero, then this Iwasawa function is non-zero and thus has only finitely many zeroes.  In the $-1$ sign case, Howard \cite{Howard} proves (using a big Heegner point) an analogous statement:\ namely, that Greenberg's conjecture on the non-vanishing of the derivatives can be verified by showing a single derivative is non-zero.

However, in \cite{Greenberg} where Greenberg states the above conjectures, he {\it nearly} states a broader conjecture writing that ``it seems reasonable to us to make an even broader conjecture" and this broader conjecture would imply that the values 
$
L(f_{\epsilon,k},\chi,j)
$
vanish finitely often as $f_{\epsilon,k}$ runs over all classical eigenforms in a Hida family, $\chi$ runs over all $p$-power conductor Dirichlet characters and $1\leq j \leq k-1$ (after one eliminates the values that are forced to vanish because of signs of functional equations).

Unlike the more restrictive conjecture, we know of no ``easy way" to verify this broader statement computationally without using the rigidity results of this paper.  Indeed, the rigidity results of this paper imply that if the above broader statement fails, then the infinite vanishing has to be very systematic---so systematic in fact that we can (and do in section \ref{sec:computations}) verify in examples that this systematic vanishing does not occur and thus only finitely many of these $L$-values vanish.

\section{Rigidity results}\label{section:rigidity}
We turn to establishing the key rigidity results for $p$-adic $L$-functions building on work of the second author in \cite{Serban:2016aa,SERBAN2022405}. These demonstrate and apply a $p$-adic formal version of the classical multiplicative Manin--Mumford conjecture (see, e.g., Zannier's book \cite{Zannierbook}). In particular, we use the following result (\cite[Theorem 1.3.]{Serban:2016aa}):
\begin{theorem}\label{thm:JTNBresult}
  Let $A=\calO_F[[X_1,\ldots, X_n]]/I$, where $\mathcal{O}_F$ denotes the ring of integers of a finite extension of $\mathbb{Q}_p$. Let $\calS$ denote the set $\{\zeta-1~\vert~ \zeta\in\mu_{p^\infty}(\overline{\IQ}_p)\}$ of torsion points of the formal multiplicative group and consider more generally the set of $n$-tuples $\calS^n$.
\begin{enumerate}
\item If the formal scheme $\Spf(A)$ contains infinitely many torsion points in $\calS^n$, then it contains a translate by a torsion point of a formal torus $\widehat{\IG}_m^k$ for some $k>0$.
\item There exists in $\Spf(A)$ a finite union $\mathcal{T}$ of translates by torsion points of formal subtori and an explicit constant $C_I>0$ depending on $I$ and our normalization of a $p$-adic absolute value $\vert-\vert_p$ such that there is $\phi\in I$ satisfying
$$\vert\phi(\zeta_1-1,\ldots, \zeta_n-1)\vert_p> C_I $$
provided $(\zeta_1-1,\ldots, \zeta_n-1)\in \calS^n\setminus (\mathcal{T}(\overline{\IQ}_p)\cap \calS^n)$.
\end{enumerate}
\end{theorem}
Note that the translates above and throughout this text are taken with respect to the multiplicative formal group law $\calF(X,Y)=(X+1)(Y+1)-1$ in each coordinate. \par
Theorem \ref{thm:JTNBresult} may be interpreted as follows:\ the torsion points $\calS^n$ are a set of \emph{special points} which may be viewed as zero-dimensional \emph{special subschemes}, where the latter constitute of translates by a torsion point of formal subtori of $\widehat{\IG}_m^n$. The theorem then states that the formal Zariski-closure of a set of special points is a special subscheme. Moreover, in the ultrametric topology it suffices for infinitely many special points to just approach a subscheme to force it to contain a special subscheme (this is a version of the Tate--Voloch conjecture \cite{TateVoloch} for $p$-adic power series rings). 
\par
We shall adapt the $p$-adic results to the context of two-variable $p$-adic $L$-functions. In particular, we shall apply Theorem \ref{thm:JTNBresult} when $n=2$ and when $I$ is the ideal generated by a branch of the $p$-adic $L$-function $L_{p,i}(w,T)$. We consider the larger set of special points at which $L_{p,i}(w,T)$ interpolates special $L$-values in the Hida family. Using that this set is closely related to the torsion points of the formal multiplicative group, we deduce in this section the main rigidity result relevant to two-variable $p$-adic $L$-functions. Recall that $\gamma$ is a topological generator of $1+q\Zp$.
\begin{prop}\label{prop:mainrigidity}
Consider a two-variable power series $L_p(w,T)\in \IZ_p[[w,T]]$. Let $\Sigma$ denote a subset of the points at which special values are interpolated, so (under suitable normalization):\ 
$$\Sigma\subset \{(\gamma^k\zeta-1,\gamma^s\zeta'-1)\in\im _{\overline{\IQ}_p}^2:\ k\in \IN, 0\leq s\leq k-1, \zeta,\zeta'\in\mu_{p^\infty}\}. $$
Assume that there exists a non-trivial character $\chi\in X_*(\widehat{\IG}_m^2)=\Hom_{frm.grp}(\widehat{\IG}_m^2,\widehat{\IG}_m)$ such that the image $\chi(\gamma^k-1,\gamma^s-1)$ over $s,k$ varying in $\Sigma$ is a finite set (for instance if $s$ is constant or $s=k/2$ in $\Sigma$). Then exactly one of the following holds:\ 
\begin{enumerate}
    \item There are only finitely many $(w,T)\in \Sigma$ such that $L_p(w,T)=0$.
    \item The power series $L_p$ vanishes along the translate of a full positive-dimensional formal subtorus of $\widehat{\IG}_m^2/\IZ_p$ by an element of $\Sigma$. 
\end{enumerate}
\end{prop}

 Here $X_*(\widehat{\IG}_m^2)=\Hom_{frm.grp}(\widehat{\IG}_m^2,\widehat{\IG}_m)$
denotes the (formal) characters on the formal torus.  In this case, these characters just correspond to taking $p$-adic integer exponents in each coordinate and the group is thus isomorphic to $\IZ_p^2$. Moreover, the formal subtori correspond to kernels of $\IZ_p$-submodules of $\IZ_p^2$ under this identification (see Lemma \ref{lemma:subtoriexplicit}).

\begin{remark}
From a rigidity point of view, requiring for some character $\chi$ that the image $\chi(\gamma^k-1,\gamma^s-1)$ is a finite set as in Proposition \ref{prop:mainrigidity} is necessary:\ indeed one can construct formal power series (see \cite[Proposition 2.11]{Serbanunlikely}) such that 
$$\varphi(\gamma^k-1,\gamma^s-1)=0 \text{ for an infinite set }S=\{(k,s)\}\subset\IZ^2,$$
and such that there is no character in $ X_*(\widehat{\IG}_m^2)$ vanishing on infinitely many elements with exponents in $S$. In particular $\varphi$ does not vanish along a subtorus translate. 
\end{remark}
Fortunately, it suffices for our purposes to mainly consider exponents on the line $s=k/2$ and the rigidity methods then allow us to conclude the main result of this section:\ 

\begin{theorem}
\label{thm:mainrigidity}  
Consider one of the finitely many branches of a two-variable $p$-adic $L$-function expanded as $L_{p,i}(w,T)\in \IZ_p[[w,T]]$. Assume $L_{p,i}$ vanishes at infinitely many critical $L$-values for the interpolated cuspidal eigenforms. Then $L_{p,i}(w,T)$ must be, possibly over a finite extension $\IZ_p[\xi]$, divisible by a factor of the form
$$\gamma^{K/2-KN}(w+1)^N-\xi(T+1)\text{ or }\gamma^{KN/2-K}(w+1)-\xi(T+1)^N$$ for some fixed $N,K\in \IZ_p$ and $\xi \in \mu_{p^\infty}$. 
\end{theorem}
\begin{proof}
First, by Proposition \ref{prop:onlycentralvanishing}, critical $L$-values only vanish at central values and thus we have
 $L_{p,i}(\gamma^k\zeta-1,\gamma^{k/2}\zeta'-1)=0$ for infinitely many triplets $(k,\zeta,\zeta')\in 2\IN\times\mu_{p^\infty}^2$. We may then apply Proposition \ref{prop:mainrigidity} to obtain that $L_{p,i}(w,T)$ vanishes along the translate of a one-dimensional formal subtorus of $\widehat{\IG}_m^2$ by a point $(\gamma^K\zeta-1,\gamma^{K/2}\zeta'-1)$ for some fixed $(K,\zeta,\zeta')\in \IZ_p\times\mu_{p^\infty}^2$. This forces the divisibility of $L_{p,i}$ by a factor of the prescribed form (see Lemma \ref{lemma:subtoritranslate}). 
\end{proof}

Before turning to the proof of Proposition  \ref{prop:mainrigidity}, we highlight that this result implies that an infinite amount of vanishing triggers vanishing of critical values along every point on a torus translate. Such a regular pattern of vanishing is not to be expected unless it occurs due to the functional equation. We exploit this in Section \ref{sec:excluding} to establish criteria for checking only finite vanishing occurs in concrete families.

\begin{proof}[Proof of Proposition \ref{prop:mainrigidity}]
Assume $L_p$ vanishes along infinitely many points in $\Sigma$. By a pigeonhole principle we easily reduce to the case when the image of the character $\chi(\gamma^k-1,\gamma^s-1)$ over $s,k$ varying in $\Sigma$ is a singleton, so that $s$ and $k$ satisfy a $\IZ_p$-linear relation. We prove a slightly stronger statement where $s,k\in \IZ_p$ and are allowed to vary freely along this line. After a possible translation and since the statement is then symmetric in both variables we may assume that
$$L_p(\gamma^k\zeta-1,\gamma^{k\cdot U}\zeta'-1)=0$$
for infinitely many triplets $(\zeta, \zeta',k)$ where $k\in\IN$, $\zeta,\zeta'\in\mu_{p^\infty}$ and $U\in \IZ_p$ is fixed. We shall use the following trick employed in \cite[Proposition 3.1]{SERBAN2022405} to conclude using Theorem \ref{thm:JTNBresult}:\ set $\Sigma_U:=\{(\zeta, \zeta',k)~\vert~ L_p(\gamma^k\zeta-1,\gamma^{kU}\zeta'-1)=0\}$. We may assume that the set of exponents $k$ in $\Sigma_U$ is infinite, since otherwise we may reduce to the first part of Theorem \ref{thm:JTNBresult}. Then by compactness, we may assume these weights $k$ converge $p$-adically to some $K\in \IZ_p$ after possibly passing to a subsequence. After a change of coordinates sending $w\mapsto \gamma^K(1+w)-1$ and $T\mapsto \gamma^{KU}(1+T)-1$ we may therefore assume the exponents $(k,kU)$ converge $p$-adically to zero. When these exponents have large enough $p$-adic valuation, one therefore obtains that the points $(\gamma^k\zeta-1,\gamma^{kU}\zeta'-1)$ $p$-adically approach the torsion points $(\zeta-1,\zeta'-1)$ of the formal group $\widehat{\IG}_m^2(\overline{\IQ}_p)^{tor}$ in the $p$-adic topology. In other words, the changes of variables yield a formal power series $\phi\in\IZ_p[[w,T]]$, together with, for any $\varepsilon>0$, a set $\Sigma_\varepsilon$ of triplets $(\zeta, \zeta',k)$ with infinitely many weights $k$, such that for a given choice of $p$-adic absolute value we have:\ 
$$\vert\phi(\zeta-1,\zeta'-1)\vert_p\leq \varepsilon \text{ for all }(\zeta, \zeta',k)\in\Sigma_\varepsilon$$
by continuity of $\phi$. Consider now the second statement of Theorem \ref{thm:JTNBresult} as applied to $n=2$ and $I=(\phi)$. Fixing a small enough $\varepsilon>0$ we conclude that either $\phi$ vanishes along a one-dimensional formal subtorus translate or the projection to the first two coordinates of $\Sigma_\varepsilon$ is finite, so that the roots of unity in $\Sigma_\varepsilon$ have bounded order $p^m$ for some $m$. In the latter case it follows that $\phi$ vanishes along $(w,T)=(\gamma^k\zeta-1, \gamma^{kU}\zeta'-1)$ for infinitely many $k$ and fixed $\zeta,\zeta'\in\mu_{p^m}$ and this set of points is therefore Zariski-dense in the formal torus translate $\Spec(\IZ_p[\zeta,\zeta' ][[w,T]]/((1+w)^U-(1+T)\zeta^{-1}\zeta')$. We conclude in either case that $\phi$ vanishes along the translate of a positive dimensional formal subtorus by a point in $\Sigma$. Tracing back through our changes of variables, the same is true for $L_p$, concluding the proof. 
\end{proof}
Finally, we record the auxiliary lemmas:
\begin{lemma}\label{lemma:subtoriexplicit}
Let $X_*(\widehat{\IG}_m^n):=\Hom_{frm.grp}(\widehat{\IG}_m^2,\widehat{\IG}_m)\cong \IZ_p^2$ denote the (formal) characters on the formal torus. Then any (closed, affine) formal subtorus of positive dimension is the vanishing locus of a $\mathbb{Z}_p$-submodule $M$ of $X_*(\widehat{\IG}_m^2)$. 
Moreover, if $\phi\in \IZ_p[[w,T]]$ vanishes along a full one-dimensional formal subtorus of $\widehat{\IG}_m^2$, then $\phi$ is divisible by $(w+1)^N-(T+1)$ or $(w+1)-(T+1)^N$ for some $N\in\IZ_p$.
\end{lemma}
\begin{proof}
For the first statement see e.g. \cite[Lemma 2.8.]{Serbanunlikely}. This then implies that if $\phi\in \IZ_p[[w,T]]$ vanishes along a full one-dimensional formal subtorus of $\widehat{\IG}_m^2$, there exist $(M,N)\in \IZ_p^2$, not both zero, such that up to coordinate swap $\phi(w,T)$ and $(w+1)^M(T+1)^N-1$ have infinitely many zeroes in common. Working up to units and restricting to the connected component of $\phi$ passing through the origin, it is easy to see that one of the $p$-adic exponents can be eliminated, and the result follows from the irreducibility of $(w+1)^N-(T+1)$ over $\IZ_p$. Irreducibility in turn follows easily from the linearity in one of the variables. 
\end{proof}
\begin{lemma}\label{lemma:subtoritranslate}
If $\phi\in \IZ_p[[w,T]]$ vanishes along a translate of a one-dimensional formal subtorus of $\widehat{\IG}_m^2$ by a point $(\gamma^{K_1}\zeta-1,\gamma^{K_2}\zeta'-1)$ for fixed $K_1,K_2\in\IZ_p$ and $\zeta,\zeta'\in \mu_{p^\infty} $, then $\phi$ is divisible (over a finite extension $\IZ_p[\xi]$) by a factor given by $$\gamma^{K_2-K_1N}(X+1)^N-\xi(Y+1)\text{ or }\gamma^{K_2N-K_1}(X+1)-\xi(Y+1)^N$$
for some fixed $N\in \IZ_p$ and $\xi\in \mu_{p^\infty}$.
\end{lemma}
\begin{proof}
By Lemma \ref{lemma:subtoriexplicit}, since $\phi$ vanishes along the translate of a formal subtorus of positive dimension, $\phi$ has infinitely many zeroes in common with a power series of the form $(\gamma^{-K_1}\zeta^{-1} (X+1))^N-\gamma^{-K_2}\zeta'^{-1}(1+X)$ or $\gamma^{-K_1}\zeta^{-1} (X+1)-(\gamma^{-K_2}\zeta'^{-1}(1+X))^N$ for some $N\in \IZ_p$. Working up to units, we may rewrite these power series as 
$$\gamma^{K_2-K_1N}(X+1)^N-\xi(Y+1)\text{ or }\gamma^{K_2N-K_1}(X+1)-\xi(Y+1)^N,$$
where $\xi$ is a fixed $p$-power root of unity given by $\zeta'^{-1}\zeta^N$ or $\zeta'^{-N}\zeta$, respectively. Since the resulting power series are again irreducible, the result follows taking Zariski closures. 

\end{proof}

\begin{remark}
\label{rmk:bigrank}
We have by and large so far assumed that the Hida-Hecke algebra $\IT$ has rank one over the Iwasawa algebra, which already covers many examples and ensures branches of the two-variable $p$-adic $L$-function can be expanded as a power series $L_{p,i}\in \IZ_p[[w,T]]$. In general, we have a finite flat map $\Spec(\IT)\to \Spec(\Lambda)$ which is \'etale at the classical weights and a branch of the two-variable $p$-adic $L$-function $\mathcal{L}_p$ may be viewed as an element of the completed group ring $\IT[[1+p\IZ_p]]\cong \IT[[T]]$ which to a prime $f_{k,\epsilon}$ above a classical weight associates the $p$-adic $L$-function $L_p(f_{k,\epsilon})$. The rigidity results of this section then again imply that any infinite amount of vanishing for $\mathcal{L}_p$ has to occur above some formal subtorus translate in $\Spec(\Lambda[[T]])$, i.e.\ along a very regular pattern. By a Zariski--closure argument, one can then for concrete $\IT$ lift this to a component of $\mathcal{L}_p$ and obtain a regular vanishing pattern that in theory could be excluded computationally. However, in practice there would be additional obstacles in computing with overconvergent modular symbols in this more general setting. 
\end{remark}

\section{Excluding  toroidal factors}\label{sec:excluding}
The rigidity results above imply that if an infinite amount of vanishing occurs in our specified set, our power series needs to be divisible by a factor of the form $D(1+X)^N-\xi(Y+1)$ or $D(1+X)-\xi(Y+1)^N$ for a constant $D\equiv 1 \pmod p$. We will refer to these factors as {\it toroidal} factors (as they arise from vanishing along a translate of a subtorus).  In order to prove the non-vanishing of the classical $L$-values interpolated by two-variable $p$-adic $L$-functions (up to a finite set), we show how such toroidal factors can be excluded computationally. 
\subsection{First method:\ low-degree coefficients}
Assume some branch of a two-variable $p$-adic $L$-function $L_{p,i}(w,T) \in \IZ_p[[w,T]]$ vanishes at the origin and we wish to exclude divisibility by a toroidal factor
(to conclude by Theorem \ref{thm:mainrigidity} that only finitely many of the classical $L$-values interpolated by $L_{p,i}(w,T)$ vanish).  The accuracy of our computer computations will yield a power series expansion for $L_{p,i}(w,T)$ modulo some power of the maximal ideal $\im=(p,X,Y)$. The one-dimensional formal subtori of $\widehat{\IG}_m^2$ sweep through the entire (co)tangent space of $\IZ_p[[X,Y]]$, however that is no longer true for $\im^i/\im^{i+1}$ and $i\geq 2$. This means that in favorable situations toroidal factors can be excluded by explicit computation. In particular, we shall use:\ 

\begin{lemma}\label{lemma:exclude}
Let $p$ be a prime and 
$$
\varphi(X,Y) = a_{X^2} X^2 + a_{XY} XY + a_{Y^2} Y^2 + \dots  \in (X,Y)^2 \cdot\IZ_p[[X,Y]].
$$
Assume
\begin{enumerate}
    \item \label{part:cond1} exactly two out of the three quadratic coefficients $a_{X^2}$, $a_{XY}$, and $a_{Y^2}$ are divisible by $p$,
    \item \label{part:cond2} $\varphi(X,0)\not\equiv 0 \mod (p,X^{p+1})$,
    \item \label{part:cond3} $\varphi(0,Y)\not\equiv 0 \mod (p,Y^{p+1})$.
\end{enumerate}
Then $\varphi$ does not vanish along a translate of a formal torus in the sense that $\varphi(X,Y)$ is not divisible by $D (X+1)^N-\xi\cdot(Y+1)$ or $D(X+1)-\xi\cdot(Y+1)^N$ for $D,N\in \IZ_p$, $D \equiv 1 \pmod{p}$, and $\xi$ a $p$-power root of unity.
\end{lemma}

\begin{proof}
We begin with the case of $D=\xi=1$ and we wish to exclude a factor of the form $(1+X)^N-(Y+1)$ up to a swapping of $X$ and $Y$.  Assuming such a factor exists, we have
\begin{align*}
\varphi(X,Y) &\equiv ((1+X)^N-(Y+1))\cdot (a+bX+cY) \pmod{(X,Y)^3} \\
 &\equiv (NX + N(N+1)/2 \cdot X^2 - Y)\cdot (a+bX+cY) \pmod{(X,Y)^3} \\
 &\equiv aNX - aY + (aN(N+1)/2 + bN )X^2 + (cN-b)XY + -cY^2 \pmod{(X,Y)^3}.
\end{align*}
As we assumed $\varphi(X,Y) \in (X,Y)^2$, we have $a=0$ and the above equation simplifies to
$$
\varphi(X,Y) \equiv bN X^2 + (cN-b)XY + -cY^2 \pmod{(X,Y)^3}.
$$
Observe that when $p\nmid N$, it is impossible that \emph{exactly} two of the three coefficients vanish modulo $p$. We may therefore assume that $p\mid N$ and we have
\begin{align*}
\varphi(X,0) &= ((1+X)^N-1) \cdot (bX + \dots) 
\equiv \left(\binom{N}{p}X^p + \dots\right) \cdot (bX + \dots) \pmod{p} \\
&\equiv 0 \pmod{(X^{p+1},p)}
\end{align*}
contradicting \eqref{part:cond2}.  The case where $X$ and $Y$ are swapped follows identically except now \eqref{part:cond3} is contradicted.

Moving on to the general case, assume we now have a product (with factors over $\IZ_p[\xi]$):\ 
$$\phi(X,Y) = (D(X+1)^N-\xi\cdot(Y+1))\cdot (a+bX+cY+\ldots).$$
The constant term of $D(X+1)^N-\xi\cdot(Y+1))$ is $D-\xi$ and vanishes if and only if $D=\xi=1$ which is the case already handled.  Thus, we may assume that $D-\xi$ is non-zero which forces $a$ to vanish (as $\phi(0,0)=0$). Thus
$$
\phi(X,Y) = (D-\xi)(bX + cY) \pmod{(X,Y)^2}
$$
which implies that $b=c=0$ as $\phi(X,Y) \in (X,Y)^2$.  But then all of the quadratic coefficients of $\phi(X,Y)$ are  divisible by $D-\xi$.  As $D \equiv 1 \pmod{p}$, we have that $D-\xi$ is not a $p$-adic unit and thus all three of these coefficients are divisible by $p$ which contradicts \eqref{part:cond1}.
\end{proof}

We will see in section \ref{sec:computations} that one can find many examples of two-variable $p$-adic $L$-functions of Hida families whose constant and linear terms vanish.  Further, we will find examples which satisfy all of the hypotheses of Lemma \ref{lemma:exclude}. However, this approach does not allow one to systematically exclude toroidal factors even in these examples.  For instance, in a sample where quadratic coefficients of $L_p(w,T)$ are equidistributed modulo $p$, the probability that exactly 2 of these coefficients are divisible by $p$ is $3(p-1)/p^3$---a ratio that tends to 0 as $p$ grows!  

\subsection{Second method:\ mod $p$ vanishing of specializations}
\label{sec:modp_check}
We now consider a different approach to computationally eliminating toroidal factors which works more generally and more systematically. 

\begin{prop}\label{prop:modpchecking}
Take $\phi(X,Y) \in \Zp[[X,Y]]$ and write $\phi(X,Y) = \sum_{j \geq 0} \phi_j(X) Y^j$. Assume
\begin{enumerate}
    \item \label{part:lambda} $\phi_\lambda(X)$ is a unit power series for some $\lambda \geq 0$,
    \item $\phi(X,Y)$ vanishes at infinitely many points of the form $(\zeta \gamma^{k-2} -1 , \zeta' \gamma^{\frac{k-2}{2}}-1)$ for $k \geq 2$ even and $\zeta,\zeta' \in \mu_{p^\infty}$.
\end{enumerate}
Then there exists $N\in \IZ_p$ and $\xi \in \mu_{p^\infty}$ with multiplicative order bounded above by $\lambda\cdot \frac{p}{p-1}$ such that 
$$
\ord_p \left( \phi(x,y) \right) \geq 1 
$$
for any $x$, $y$ in the open unit disc such that 
$$
(x,y) \equiv (\zeta-1,\xi^{-1} \zeta^N -1) 
\text{~~or~~} (\xi \zeta^N-1,\zeta -1) \pmod{p}
$$
and for \textbf{any} $\zeta \in \mu_{p^\infty}$.
\end{prop}

\begin{proof}
By Theorem \ref{thm:mainrigidity}, since $\phi(X,Y)$ has infinitely many zeroes of the form $(\zeta \gamma^{k-2} -1 , \zeta' \gamma^{(k-2)/2})$, we have that $\phi(X,Y)$ is divisible by a toroidal factor of one of the following two forms:
$$
 \gamma^A (1+X)^N - \xi (1+Y) \text{~~or~~}
 \gamma^A (1+X) - \xi (1+Y)^N. 
$$
As $\gamma \equiv 1 \pmod{p}$, the first of these factors is divisible by $p$ after evaluating at $(x,y) \equiv (\zeta-1,\xi^{-1} \zeta^N -1) \pmod{p}$ while the second of these factors is divisible by $p$ after evaluating at $(x,y) \equiv (\xi \zeta^N-1,\zeta-1) \pmod{p}$.

So all that remains is to verify the bound on the multiplicative order of $\xi$.  To this end, note that if $\xi$ has order $p^n$, then the divisibility of $\phi(X,Y) \in \Zp[[X,Y]]$ by the above toroidal factor implies the divisibility by analogous toroidal factors obtained by replacing $\xi$ with any of its Galois conjugates.  In particular, we obtain $(p-1)p^{n-1}$ distinct toroidal factors which implies that $\phi(z,Y)$, for a fixed $z$ in the open disc unit, has at least $(p-1)p^{n-1}$ zeroes (in $Y$).  But our assumption that $\phi_\lambda(X)$ is a unit power series implies that $\phi(z,Y)$ has at most $\lambda$ zeroes.  Thus $\lambda \geq (p-1)p^{n-1}$ which implies $\lambda \cdot \frac{p}{p-1} \leq p^n$ as claimed.
\end{proof}

Taking $\phi(X,Y)$ in Proposition \ref{prop:modpchecking} to be $L_{p,i}(w,T)$, a branch of a two-variable $p$-adic $L$-function, gives a computational method for proving that only finitely many classical $L$-values in a Hida family vanish.  Indeed, first note that assumption \eqref{part:lambda} holds when the $\mu$-invariant of any specialization of the family vanishes (see \cite[Theorems 3.7.5 and 3.7.7]{EmertonPollackWeston}).
Further, if infinitely many classical $L$-values interpolated by $L_{p,i}(w,T)$ vanish, then Proposition \ref{prop:modpchecking} implies that for each $n$ (as long as $p^n \geq \lambda \cdot \frac{p}{p-1}$), we have
\begin{equation}
\label{eqn:ineq}
\ord_p(L_{p,i}(\zeta-1,\zeta'-1)) \geq 1
\end{equation}
for some $\zeta,\zeta' \in \mu_{p^n}$ with at least one of $\zeta$ and $\zeta'$ primitive.  In particular, by computing finitely many values of the two-variable $p$-adic $L$-function, one can try to verify that the inequality in \eqref{eqn:ineq} fails for all such $\zeta$ and $\zeta'$.  If successful, such a computation would imply that there are only finitely many classical $L$-values vanishing in the family.  If unsuccessful, one could simply repeat the computation with a larger value of $n$.  Such computations will be described in detail in section \ref{sec:computations}.

We note though that there are cases when two-variable $p$-adic $L$-functions have toroidal factors.  As mentioned in section \ref{sec:back_nonvanish}, if $i$ is congruent to $r$ or $r+\frac{p-1}{2}$ modulo $p-1$, then $L_{p,i}(w,T)$ has a functional equation which relates this power series to itself.  (Recall that $r$ is defined by our Hida family as the unique residue class mod $p-1$ such that all classical specializations of the Hida family without nebentype at $p$ have weights congruent to $r+2$ mod $p-1$.)  

When the sign of the functional equation in this case is $-1$, then $L_{p,i}(w,T)$ is divisible by $(1+w)-(1+T)^2$ which is indeed a toroidal factor.  In this case, Proposition \ref{prop:modpchecking} cannot be successfully applied to $L_{p,i}(w,T)$ to deduce that only finitely many $L$-values in the family vanish (as this is simply not true!).  

Instead though, one can apply this Proposition to $\ds \frac{L_{p,i}(w,T)}{(1+w)-(1+T)^2}$.  
In this case, if infinitely many classical $L$-values interpolated by $L_{p,i}(w,T)$ vanish (apart from the ones forced to vanish for functional equation reasons), then Proposition \ref{prop:modpchecking} implies that 
\begin{equation}
\label{eqn:ineq2}
\ord_p\left(\frac{L_{p,i}(\zeta-1+p,\zeta'-1)}{\zeta+p-(\zeta')^2}\right) \geq 1
\end{equation}
for some $\zeta,\zeta' \in \mu_{p^n}$ with at least one of $\zeta$ and $\zeta'$ primitive.  Here the additive factor of $p$ is included to avoid the zeroes of $(1+w)-(1+T)^2$.  As before, by computing finitely many values of $L_{p,i}(w,T)$, one could try to disprove the inequalities in \eqref{eqn:ineq2}.

\section{Computations}
\label{sec:computations}
\subsection{First method:\ low-degree coefficients}

To find examples of two-variable $p$-adic $L$-functions such that $L_{p,0}(w,T) \in (w,T)^2$---which is required for Lemma \ref{lemma:exclude}---consider a weight 2 newform $f_2$ of level $Np$ such that $a_p(f_2)=1$ and whose complex functional equation has sign $-1$.  In this case, the one-variable $p$-adic $L$-function of $f_2$ has a trivial zero which causes the sign of its $p$-adic functional equation to be $+1$.  In particular, $L_{p,0}(f_2,T)$ is divisible by $T^2$.  Consider the Hida family through $f_2$ which we assume to have rank 1 over weight space and write
$$
L_{p,0}(w,T) = a_1 + a_w w + a_T T + \dots.
$$
Then specializing to weight 2 (i.e.\ setting $w=0$) yields a power series divisible by $T^2$.  In particular, $a_1=a_T=0$.  In this case it is automatically then true that $a_w=0$ (see \cite[pg.\ 373]{BD-Hida}) and we have  $L_{p,0}(w,T) \in (w,T)^2$.

To find such an example, set $p=3$ and consider the elliptic curve E15a1. Let $f_2$ be the corresponding weight 2 newform and consider the Hida family through $f_2$ (which is indeed rank 1 over weight space).  This example is not directly interesting from the optics of this paper:\ indeed $\lambda(f_2)=0$ and the two-variable $p$-adic $L$-function of this Hida family is a unit and thus has no zeroes.  Also, $a_p(f_2)=-1$ which is the reason why these $L$-functions do not have any trivial zeroes.   Nonetheless, we can make these families interesting by twisting them by certain quadratic characters which:\ (a) force $a_p=1$ and (b) force $\lambda>0$.  As a matter of notation, let's write $L_{p,i}(\chi_D,w,T)$ for the two-variable $p$-adic $L$-function of the Hida family twisted by $\chi_D$, the quadratic character of conductor $D$.

The first interesting twist arises from twisting by $\chi_{29}$.  In this case, the $\lambda$-invariant is constantly equal to 2 in the Hida family.  To compute the corresponding two-variable $p$-adic $L$-function, we first use the algorithms of \cite{ExplicitOverconvergent} to compute a family of overconvergent modular symbols corresponding to this Hida family.  In the terminology of \cite{ExplicitOverconvergent}, we computed modulo $\widetilde{\Fil}^{30}$.  With this family in hand, one can then compute an approximation to any branch of the $p$-adic $L$-function (see \cite[section 5.4]{ExplicitOverconvergent}).  We note that from this one family, one can compute $p$-adic $L$-functions for {\it any} quadratic twist and so one does not need to recompute new families as one varies the twist.

These computations yield
$$
L_{3,0}(\chi_{29},w,T) = (3 \cdot 1243204 + O(3^{14})) w^2 + (148430 + O(3^{11})) wT + (28717 + O(3^{10}))T^2 + \cdots
$$
and unfortunately only 1 of the three quadratic coefficients are multiples of 3 and Lemma \ref{lemma:exclude} does not apply.

However, the next interesting twist is by $\chi_{41}$ and in this case
$$
L_{3,0}(\chi_{41},w,T) = (4540910 + O(3^{14})) w^2 + (3^2 \cdot 1211 + O(3^{11})) wT + (3^2 \cdot 5350 + O(3^{10}))T^2 + \cdots
$$
Here we see that exactly 2 of the 3 quadratic coefficients are multiples of 3.  Moreover, as the coefficient of $w^2$ is a unit, this coefficient witnesses the fact that $L_{3,0}(\chi_{41},w,0)$ is not 0 modulo $(3,w^4)$.  Further, the coefficient of $T^4$ (not shown above) is 1 modulo 3 and this witnesses the fact that $L_{3,0}(\chi_{41},0,T)$ is not 0 modulo $(3,T^4)$.  In particular, Lemma \ref{lemma:exclude} combined with Theorem \ref{thm:mainrigidity} imply that only finitely many classical $L$-value in our Hida family vanish.

\subsection{Second method:\ mod $p$ vanishing of specializations}
\label{sec:modp_comp}

We now consider the method of ruling out toroidal factors described in section \ref{sec:modp_check} which we will see works much more generally and systematically than the method of the previous section.  To demonstrate this, we consider three explicit Hida families all of which are rank 1 over the Iwasawa algebra.

\begin{enumerate}
    \item $p=3$, $N=5$, $r=0$
    \item $p=5$, $N=3$, $r=0$
    \item $p=11$, $N=1$, $r=0$
\end{enumerate}

In all three cases $S_2(\Gamma_0(Np))$ is 1-dimensional and so simply specifying $p$, $N$, and $r$ determines the Hida family.  The first two examples specialize in weight 2 to the modular form which corresponds to the elliptic curve $X_0(15)$.  The third example specializes to the modular form corresponding to $X_0(11)$ in weight 2 and to $\Delta$ in weight 12 ($11$-stabilized to $\Gamma_0(11)$).

For all three cases, we considered a range of quadratic twists with discriminant prime to $pN$, and for each fixed twist, the method worked and we established that only finitely many $L$-values of the form $L(f_{k,\epsilon}\otimes \chi_D, \chi,k/2)$ vanish apart from the $L$-values that are forced to vanish from the sign of their functional equation.  In particular, in the $p=3$ case, we considered twists by $\chi_D$ with $|D| < 350$ while in the $p=5$ and $p=11$ case, we considered twists with $|D| < 500$. The results are summarized in Theorem \ref{thm:maincomputational}.

The method is essentially the same in each case and so we illustrate it in a couple of examples.

\begin{example}($p=3$, $N=5$)
\label{ex:3}
As described in the previous section, we first use the algorithms of \cite{ExplicitOverconvergent} to compute a family of overconvergent modular symbols corresponding to our given Hida family.  In the terminology of \cite{ExplicitOverconvergent}, we computed modulo $\widetilde{\Fil}^{M}$ for varying $M$ depending on the accuracy we were seeking.  But for each fixed $M$, one only needs to compute this family of symbols one time to compute the $p$-adic $L$-function of any twist of the family.  

As before, we denote the $i$-th branch of the $p$-adic $L$-function twisted by $\chi_D$ as $L_{p,i}(\chi_D,w,T)$.  We note that working with overconvergent symbols modulo $\widetilde{\Fil}^{30}$ yields an approximation of $L_{p,i,D}(w,T)$:\ namely, we only recover terms with total degree less than 9 and, moreover, the accuracy of the known coefficients lessens as the total degree increases.

Let's take the case where we twist by the quadratic character of conductor $29$---the example we failed to verify in the previous section.  We compute both branches of the two-variable $p$-adic $L$-function:\  $L_{3,0}(\chi_{29},w,T)$ and $L_{3,1}(\chi_{29},w,T)$---that is, we take $i=0$ and $i=1$.  
For $i=1$, the resulting branch of the two-variable $p$-adic $L$-function is a unit and so never vanishes.  Thus there is nothing to check in this case.  

For $i=0$, we see that the $\mu$-invariants in this case are all 0 and the constant value of the $\lambda$-invariant is 2.  By Proposition \ref{prop:modpchecking}, if infinitely many classical $L$-values in this family vanish, then there is some $N \in \Z_3$ and some $\xi \in \mu_p$ such that for any $\zeta \in \mu_{3^\infty}$, we have that either 
$$
L_{3,0}(\chi_{29},\zeta-1,\xi^{-1} \zeta^N-1) \text{~~or~~} L_{3,0}(\chi_{29},\xi \zeta^N-1, \zeta-1)
$$
has valuation greater than or equal to 1.  To show that this does {\bf not} occur, we fix a primitive $\zeta \in \mu_{3^m}$ for $m \geq 1$ and aim to show that 
\begin{equation}
\label{eqn:values}
L_{3,0}(\chi_{29},\zeta-1,\zeta'-1) \text{~~and~~} L_{3,0}(\chi_{29},\zeta'-1,\zeta-1)
\end{equation}
all have valuation strictly less than 1 as $\zeta'$ varies in $\mu_{3^m}$.  Indeed, this computation suffices since $\xi \zeta^N$ and $\xi^{-1} \zeta^N$ both live in  $\mu_{3^m}$.

Beginning with $m=1$, we fix any primitive $\zeta \in \mu_3$ and compute the valuations of our approximations to the expressions in (\ref{eqn:values}) as $\zeta'$ varies over $\mu_3$.  In this case, there is a valuation of $5/2$ and as this valuation is greater than or equal to 1, we need to try a higher value of $m$.  Moving on to $m=2$, we fix a primitive $\zeta \in \mu_9$ and now vary over $\zeta' \in \mu_9$ repeating the analogous computation.  Again a valuation greater than 1 appears and we must increase $m$.  Taking $m=3$, fixing a primitive $\zeta \in \mu_{27}$, and varying over $\zeta' \in \mu_{27}$ yields a maximum valuation of $7/18$, finally less than 1.

At this point, we must take care that our approximations were close enough to ensure that the valuations we compute are the true valuations of the terms in \eqref{eqn:values}.  In this case, our approximation of $L_{3,0}(\chi_{29},w,T)$ only contains terms of total degree 8 and less.  Thus evaluating at $(\zeta-1,\zeta'-1)$ yields a value that is accurate modulo $p^{9/18}=p^{1/2}$ as $\zeta-1$ has valuation $1/18$ and $\zeta'-1$ has valuation at least $1/18$.\footnote{To know the accuracy of our approximation we need to take into account both the tail of the series being ignored and the accuracy of the coefficients that we have computed.  However, in every example we computed, the error caused by the tail of the series always outweighed the errors in coefficients.}   In particular, we know that the computed valuations are all correct as they are all less than or equal to $7/18 < 1/2$.

Finally, as the maximal valuation $7/18$ is strictly less than 1, we deduce by Proposition \ref{prop:modpchecking} that only finitely many classical $L$-values vanish in this family.
\end{example}

\begin{example} ($p=5$ and $N=3$)
When $p=5$ there are 4 branches of the $p$-adic $L$-function.  Let's look at the case where we twist by a quadratic character of conductor 29.  Computing with a family of overconvergent modular symbols accurate modulo $\wt{\Fil}^{5}$ yields $L_{5,i}(\chi_{29},w,T)$ computed up to terms with total degree 3.

For $i=1$ and $i=3$, the resulting branches are units and thus have no zeroes to worry about.  For the $i=0$ branch, the resulting $p$-adic $L$-function has $\mu$-invariant constantly equal to 0 and $\lambda$-invariant constantly equal to 2. Proceeding as in Example \ref{ex:3}, we fix a primitive $\zeta \in \mu_5$ and compute
\begin{equation}
L_{5,0}(\chi_{29},\zeta-1,\zeta'-1) \text{~~and~~} L_{5,0}(\chi_{29},\zeta'-1,\zeta-1)
\end{equation}
as $\zeta'$ varies in $\mu_5$.  In this case, the maximum valuation of our approximation to these values is $1/2$.  Since our approximation of $L_{5,i}(\chi_{29},w,T)$ only contains terms of total degree at most 3 and we are evaluating $w$ and $T$ at numbers with valuation at least $1/4$, our results are accurate modulo $p^{4/4} = p^1$ and the computed maximum valuation of $1/2$ is correct.  Further, since $1/2 < 1$, by Proposition  \ref{prop:modpchecking}, we deduce that only finitely many classical $L$-values in this family vanishes.

Finally, we study $L_{5,2}(\chi_{29},w,T)$---the $i=2$ branch of this $p$-adic $L$-function.  This case is significantly more complicated as the sign of the functional equation is $-1$ and $L_{5,2}(\chi_{29},w,T)$ is divisible by $(1+w)-(1+T)^2$ and we now need to apply Proposition \ref{prop:modpchecking} to $\ds \frac{L_{5,2}(\chi_{29},w,T)}{(1+w)-(1+T)^2}$.  Further, this divisibility by a toroidal factor causes the valuations we encounter to be higher which necessitates us to use more accurate modular symbols.

To this end, we use a family of overconvergent modular symbols accurate modulo $\wt{\Fil}^{46}$ which yields $L_{5,2}(\chi_{29},w,T)$ up to terms with total degree 27.  If infinitely many classical $L$-values vanish (apart from the ones forced to vanish by the sign of the functional equation), Proposition \ref{prop:modpchecking} would guarantee that there is some $\xi \in \mu_5$ and $N \in \Z_5$ such that for any fixed $\zeta \in \mu_{5^n}$, one of 
$$
\frac{L_{5,2}(\chi_{29},\zeta-1+p,\xi^{-1} \zeta^N-1)}{\zeta-\xi^{-2}\zeta^{2N}+p}
\text{~~or~~} 
\frac{L_{5,2}(\chi_{29},\xi \zeta^N-1+p, \zeta-1)}{\xi\zeta^N-\zeta^2+p}
$$
has valuation greater than or equal to 1.  Here we insert an additive factor of $p$ to avoid the zeroes of $(1+w)-(1+T)^2$.  To verify that this infinite vanishing does {\bf not} occur, we aim to show that the valuations of
$$
\frac{L_{5,2}(\chi_{29},\zeta-1+p,\zeta'-1)}{\zeta-(\zeta')^2+p} \text{~~and~~} \frac{L_{5,2}(\chi_{29},\zeta'-1+p,\zeta-1)}{\zeta'-\zeta^2+p}
$$
are all strictly less than 1 for some fixed $\zeta \in \mu_{p^m}$ ($m \geq 1$) and all $\zeta' \in \mu_{p^m}$.

We take $\zeta$ a primitive 25-th root of unity.  Computing the numerators of all of the above fractions shows that the maximal valuation occurs at $(\zeta-1+p,\zeta^2-1)$ and $(\zeta^{13}-1+p,\zeta-1)$ and this maximal valuation is $6/5$.  As $L_{5,2}(\chi_{29},w,T)$ contains terms up to total degree 25, and we are evaluating at values of $w$ and $T$ with valuation at least $1/20$, our values are known modulo $p^{26/20}$.  Fortunately, $6/5 < 26/20$ guaranteeing that all of our valuation computations are correct.

We then subtract away the valuations of the denominators of the above fractions (which we know exactly) and discover the the maximal valuation of all of these fractions is $1/5$.  As $1/5<1$, from Proposition \ref{prop:modpchecking}, we deduce that only finitely many classical $L$-values in this family vanish (apart from the ones forced to vanish from signs of functional equations).

One might wonder why we didn't use the family of overconvergent modular symbols modulo $\wt{\Fil}^{46}$ throughout this entire example.  The reason is simple:\ using a family modulo $\wt{\Fil}^{5}$ to compute all of the branches of the $p$-adic $L$-functions takes a few seconds.  Working modulo $\wt{\Fil}^{46}$ requires several hours just to compute a single branch of the $p$-adic $L$-functions.  So we only used families with high accuracy when we really needed to.
\end{example}

\begin{example} ($p=11$ and $N=1$)
For $p=11$, there are 10 branches of the $p$-adic $L$-function.  Let's look in detail at the case where we twist by the quadratic character of conductor 129.  In this case, we used a family of overconvergent modular symbols modulo $\wt{\Fil}^5$ yielding $L_{11,i}(\chi_{129},w,T)$ with terms up to total degree 3. 

Of the 10 branches, there are 5 units (corresponding to $i=2,4,5,6,8$) which we don't need to consider because they have no zeroes.  For $i=0$, we have that the constant value of $\lambda$ is 1 and in this case the sign of the functional equation is $-1$.  Thus, $L_{11,0}(\chi_{129},w,T)$ is divisible by $(1+w)-(1+T)^2$ and this factor accounts for all of its zeroes (as $\lambda=1$) and again we have accounted for all of the zeroes.  

For each of the remaining branches ($i=1,3,7,9$), the $\mu$-invariant vanishes and the constant value of the $\lambda$-invariant is 1.  But this time the functional equation is not accounting for these zeroes.  Indeed, the functional equation for these indices simply relates two distinct branches of the $p$-adic $L$-function and one does not obtain any forced vanishing.  To proceed, we must use the methods of this paper to ensure that only finitely many classical $L$-values vanish in this family. 

Following the method of the previous two examples, we fix a primitive $\zeta \in \mu_{11}$ and want to show that the valuations of 
$$
L_{11,i}(\chi_{129},\zeta-1,\zeta'-1) \text{~~and~~} L_{11,i}(\chi_{129},\zeta'-1,\zeta-1)
$$
are all strictly less than 1 as $\zeta'$ varies over $\mu_{11}$.  For each branch, the maximum valuation of our approximation to these values is $1/5$.  As we know $L_{11,i}(\chi_{129},w,T)$ up to terms of total degree 3 and we are evaluating $w$ and $T$ at numbers with valuation at least $1/10$, our specializations are accurate modulo $p^{2/5}$.  As $1/5 < 2/5$, the valuations we computed are the true valuations and as $1/5 < 1$, Proposition \ref{prop:modpchecking} implies that only finitely many of the classical $L$-values interpolated by these remaining 4 branches vanish.
\end{example}
\nocite{}
\bibliography{padicbiblio.bib}

\end{document}